\documentclass[10pt]{amsart}
\usepackage{amssymb,enumerate}
\usepackage{amsmath}
\usepackage{graphics}
\usepackage{cjmb5}

\begin{document}

\title{A hybrid scheme for fixed points of a countable family of generalized nonexpansive-type maps and finite families of variational inequality and equilibrium problems, with applications}
\author{Markjoe O. Uba}
\address{
 University of Nigeria \newline
\indent Department of Mathematics\newline
\indent Nsukka - Onitsha Rd, 410001 Nsukka, Nigeria}

\email{markjoeuba@gmail.com}

\author{Maria A. Onyido}
\address{
 Northern Illinois University \newline
\indent Department of Mathematical Sciences, \newline
\indent DeKalb, IL 60115, United States}
\email{mao0021@auburn.edu}

\author{Cyril I. Udeani}
\address{
 Comenius University in Bratislava \newline
\indent Faculty of Mathematics, Physics and Informatics, \newline
\indent Mlynská dolina F1, 842 48 Bratislava, Slovak Republic}
\email{cyrilizuchukwu04@gmail.com}

\author{Peter U. Nwokoro}
\address{
 University of Nigeria \newline
\indent Department of Mathematics \newline
\indent Nsukka - Onitsha Rd, 410001 Nsukka, Nigeria}
\email{peter.nwokoro@unn.edu.ng }
\setcounter{page}{1}
\corauthor{Maria A. Onyido; mao0021@auburn.edu}
\begin{unabstract}
Let $C$ be a nonempty closed and convex subset of a uniformly smooth and uniformly convex real Banach space $E$ 
with dual space $E^*$. We present a novel hybrid method for finding a common solution of a family of equilibrium problems, a common solution of a family of variational inequality problems and a common element of fixed points of a family of a general class of nonlinear nonexpansive maps. The sequence of this new method is proved to converge strongly
to a common element of the families. Our theorem and its applications complement, generalize, and extend various results in literature.
\end{unabstract}

\maketitle
\pagestyle{myheadings}
\markboth{Markjoe O. Uba, Maria A. Onyido, Cyril I. Udeani, and Peter U. Nwokoro}{Equilibrium, variational inequality, and fixed point problems}

\section{Introduction}
\noindent Let $E$ be a real Banach space with topological dual $E^*$. Let $C \subset{E}$ be closed and convex with $JC$ also  closed and convex, where $J$ is the normalized duality map (see definition \ref{ndm}). The variational inequality problem, which has its origin in the 1964 result of Stampacchia \cite{stamp}, has engaged the interest of researchers in the recent past (see, e.g., \cite{zegey, zegeyofo} and many others). This is concerned with the following: For a monotone operator $A : C \to E$,  find a point $x^* \in C$ such that
\begin{equation}\label{vip}
\langle y - x^*, Ax^* \rangle \ge 0 ~~for~~all~~y \in C.     
\end{equation}
The set of solutions of (\ref{vip}) is denoted by $VI(C,A)$. This problem, which plays a crucial role in nonlinear analysis, is also related to fixed point problems, zeros of nonlinear operators, complementarity problems, and convex minimization problems (see, for example, \cite{ Dong10, Saewan13}).

\noindent A related problem is the equilibrium problem, which has been studied by several researchers and is mostly applied in solving optimization problems (see \cite{Blum}). For a map $f : C \to E $, the equilibrium problem is concerned with finding a point $x^* \in C$ such that 
\begin{equation}\label{fpp}
    f(x^*,y) \ge 0  ~~for ~~all ~~ y \in C.
\end{equation}
The set of solutions of (\ref{fpp}) is denoted by $EP(f)$. The variational inequality and equilibrium problems are special cases of the so-called generalized mixed equilibrium problem (see \cite{PeYa}).
Another related problem is the fixed point problem. For a map $T : D(T)\subset E \to E$, the fixed points of $T$ are the points $x^* \in D(T) ~~such ~~ that ~~ Tx^* = x^*$.
 Recently, owing to the need to develop methods for solving fixed points of problems for functions from a space to its dual, a new concept of {\it fixed points for maps from a real normed space $E$ to its dual 
space $E^*$}, called $J-$fixed point has been introduced and studied (see  \cite{cidu, BL, HZeg}).

\noindent With this evolving fixed point theory, we study the $J-$fixed points of certain maps and the following equilibrium problem. 
Let $f : JC \times JC \rightarrow \mathbb{R}$ be a bifunction. The equilibrium problem for $f$ is
finding 
\begin{eqnarray}\label{111k}
x^*\in C~ such~ that~ f(Jx^*,Jy) \geq 0, \forall~ y \in C. 
\end{eqnarray}
We denote the solution set of (\ref{111k}) by $EP(f)$.
Several problems in physics, optimization and economics reduce to finding a solution of (\ref{111k})
(see, e.g., \cite{comb, zegey} and 
the references in them). Most of the equilibrium problems studied in the past two decades centered on their existence and applications (see, e.g.,  \cite{Blum, comb} ). However, recently, several researchers have started working on finding
approximate solutions of equilibrium problems and their generalizations
(see, e.g.,  \cite{b2, zegeyofo}). 
Not long ago, some researchers investigated the problem of establishing a common element in the solution set of an equilibrium problem, fixed point of a family of nonexpansive maps and solution set of a variational inequality problem for different classes of maps (see \cite{zesh} and references therein).\\

\noindent In this paper, inspired by the above results especially the works in \cite{ceez, uoo, zesh}, we present an algorithm for finding a common element of the fixed point of an infinite family of generalized $J_*-$nonexpansive maps, the solution set of the variational inequality problem of a finite family of continuous monotone maps and the solution set of the equilibrium point of a finite family of bifunctions satisfying some given conditions. Our results complement, generalize and extend results in \cite{ kuman, Qsu, nakajok, zesh} (see the section on conclusion) and other recent results in this direction.
\noindent It is worth noting that very recently, the authors in \cite{ceez} introduced a new class of maps which they called {\it relatively weak $J-$nonexpasive} and developed an algorithm for approximating a common element of the $J-$fixed point of a countable family of such maps and zeros of some other class of maps in certain Banach spaces. Previously, maps with similar requirements as these {\it relatively weak $J-$nonexpasive} maps have also been studied in \cite{coeu} where they were called {\it quasi$-\phi-J-$nonexpansive}. We observe that these two sets of maps ({\it relatively weak $J-$nonexpasive} and {\it quasi $-\phi-J-$ nonexpansive}) coincide in definition with the $J_*-$nonexpansive maps in our results. 
\section{Preliminaries}
\noindent In this section, we present definitions and lemmas used in proving our main results.

\begin{definition}\label{ndm}{\bf (Normalized duality map)} \rm
 The map $J:E\rightarrow 2^{E^*}$ defined by
$$Jx:=\big \{x^*\in E^*:\big <x,x^*\big >=\|x\|.\|x^*\|,~\|x\|=\|x^*\|\big \}$$ is called the {\it normalized duality map} on $E$.
\end{definition}
 
\noindent It is well known that if $E$ is smooth, strictly convex and reflexive then $J^{-1}$ exists 
(see e.g., \cite{tak}); 
$J^{-1}:E^{*}\rightarrow E$ is the normalized duality mapping on $E^{*}$,
and $J^{-1}=J_{*}, ~JJ_{*}=I_{E^{*}}$ and $J_{*}J =I_{E}$, where $I_{E}$ and $I_{E^{*}}$ are the identity maps on
 $E$ and $E^{*}$, respectively.  A well known property of $J$ is, see e.g., \cite{Io, tak}, if $E$ is uniformly smooth, then $J$ is uniformly continuous on bounded subsets of $E$. 

\begin{definition}{\bf (Lyapunov Functional)} \rm\cite{b1, b2}
Let $E$ be a smooth real Banach space with dual $E^*$. The {\it Lyapounov functional} $\phi:E\times E\to\mathbb{R}$, is defined by
\begin{eqnarray}\label{Lya}
 \phi(x,y)=\|x\|^2-2\langle x,Jy\rangle+\|y\|^2,~~\text{for}~x,y\in E,
\end{eqnarray}
where $J$ is the normalized duality map. If $E=H$, a real Hilbert space, then
equation (\ref{Lya}) reduces to $\phi(x,y)=\|x-y\|^2$ for $x,y\in H.$ Additionally, 
\begin{eqnarray}\label{fi}
 (\|x\|-\|y\|)^2\leq \phi(x,y)\leq(\|x\|+\|y\|)^2~~\text{for}~x,y\in E.
\end{eqnarray}
\end{definition}

\begin{definition}\label{gen nonexp}{\bf (Generalized nonexpansive)} \rm \cite{kst, kota}
Let $C$ be a nonempty closed and convex subset of a real Banach space $E$ and $T$ be a map from $C$ to $E$. The map $T$ is called 
{\it generalized nonexpansive} if $F(T):=\{x\in C: Tx=x\}\neq \emptyset$ and $\phi(Tx,p)\leq \phi(x,p)$ for all $x\in C, p\in F(T)$. 
\end{definition}
\begin{definition}{\bf (Retraction)} \rm \cite{kst, kota}
A map $R$ from $E$ onto $C$ is said to be a retraction if $R^{2}=R$. The map $R$ is said to be {\it sunny}  if $R(Rx+t(x-Rx))=Rx$ 
for all $x\in E$ and $t\leq 0$. 
\end{definition}

\noindent A nonempty closed subset $C$ of a smooth Banach space $E$ is said to be a {\it sunny generalized nonexpansive retract} of $E$ if there exists 
a sunny generalized nonexpansive retraction $R$ from $E$ onto $C$.

\noindent {\bf NST-condition.}
Let $C$ be a closed subset of a Banach space $E$. Let $\{T_{n}\}$ and $\Gamma$ be two families of generalized nonexpansive maps of $C$ into 
$E$ such that $\cap_{n=1}^{\infty}F(T_{n})=F(\Gamma)\neq \emptyset,$ where $F(T_{n})$ is the set of fixed points of $\{T_{n}\}$ and $F(\Gamma)$
is the set of common fixed points of $\Gamma$. 

\begin{definition} \rm \cite{kst}
The sequence  $\{T_{n}\}$ satisfies the NST-condition (see e.g., \cite{nakajostak}) with $\Gamma$ if for each bounded sequence $\{x_{n}\}\subset C$,
$$\lim_{n\rightarrow \infty}||x_{n}-T_{n}x_{n}||=0 \Rightarrow  \lim_{n\rightarrow \infty}||x_{n}-Tx_{n}||=0, ~for ~all~T\in \Gamma.$$
\end{definition}

\begin{remark}\label{rmk1} \rm
  If $\Gamma =\{T\}$ a singleton, $\{T_{n}\}$ satisfies the NST-condition with $\{T\}$. If $T_{n}=T$ for all $n\geq 1$, 
then, $\{T_{n}\}$ satisfies the NST-condition with $\{T\}$.
\end{remark}

\noindent Let $C$ be a nonempty closed and convex  subset of a uniformly smooth and uniformly convex real Banach space $E$ with dual space $E^*$.
Let $J$ be the normalized duality map on $E$ and $J_{*}$ be the normalized duality map on $E^*$. Observe that under this setting, $J^{-1}$ exists 
and $J^{-1}=J_{*}$. With these notations, we have the following definitions.

\begin{definition} {\bf (Closed map)} \cite{uoo} \rm
 A map $T:C\rightarrow E^*$ is called {\it $J_{*}-$closed} if $(J_{*}\circ T) : C\rightarrow E$ is a closed map, i.e., if $\{x_{n}\}$ is a sequence in $C$ 
such that $x_{n}\rightarrow x$ and  $(J_{*}\circ T)x_{n}\rightarrow y$, then $(J_{*}\circ T)x =y$. 
\end{definition}

\begin{definition}{\bf ($J-$fixed Point)} \cite{cidu} \rm
A point $x^*\in C$ is called a {\it $J-$fixed point of $T$} if $Tx^*=Jx^*$. The set of $J-$fixed points of $T$ will be denoted by $F_{J}(T)$.
\end{definition}

\begin{definition}{\bf (Generalized $J_*-$nonexpansive)} \cite{uoo} \rm
 A map $T:C\rightarrow E^*$ will be called {\it generalized $J_{*}-$nonexpansive} if $F_{J}(T)\neq \emptyset$, and 
$\phi (p, (J_{*}\circ T)x) \leq \phi(p,x)$ for all $x\in C$ and for all $p\in F_{J}(T)$.
\end{definition}

\begin{remark}\label{rmk2}\rm
Exampes of generalized $J_*-$nonexpansive maps in Hilbert and more general Banach spaces were given in \cite{ceez, uoo}.
\end{remark}

\noindent Let $C$ be a nonempty closed subset of a smooth, strictly convex and reflexive Banach space $E$ such that $JC$ is closed and convex. For solving
the equilibrium problem, let us assume that a bifunction $f:JC\times JC\rightarrow \mathbb{R}$ satisfies the following conditions:
\begin{itemize}
 \item[(A1)]  $f(x^*,x^*)=0$ for all $x^*\in JC$;
 \item[(A2)]  $f$ is monotone, i.e. $f(x^*,y^*)+f(y^*,x^*)\leq0$ for all $x^*,y^*\in JC$;
 \item[(A3)]  for all $x^*,y^*,z^*\in JC$, $\limsup_{t\downarrow0} f(tz^*+(1-t)x^*,y^*)\leq f(x^*,y^*)$;
 \item[(A4)]  for all $x^*\in JC$, $f(x^*,\cdot)$ is convex and lower semicontinuous. 
\end{itemize}

\noindent With the above definitions, we now provide the lemmas we shall use.

\begin{lemma} \cite{Zh}\label{zhang} \rm
Let $E$ be a uniformly convex Banach space, $r > 0$ be a positive number, and $B_r(0)$ be a closed ball of $E$. For any given points $\{ x_1, x_2, \cdots , x_N \} \subset B_r(0)$ and any given positive numbers $\{ \lambda_1, \lambda_2, \cdots , \lambda_N \}$ with $\sum_{n = 1} ^{N} \lambda_n = 1,$ there exists a continuous strictly increasing and convex function $g : [0, 2r) \to [0, \infty)$ with $g(0) = 0$ such that, for any $i,j \in \{ 1,2, \cdots N \}, \; i < j,$
\begin{equation}
    \| \sum_{n = 1} ^{N} \lambda_n x_n \|^2 \le \sum_{n = 1} ^{N} \lambda_n\|x_n\|^2 - \lambda_i \lambda_j g(\|x_i - x_j\|).
\end{equation}
\end{lemma}

\begin{lemma} \cite{b2}\label{man} \rm
 Let $X$ be a real smooth and uniformly convex Banach space, and let $\{x_n\}$ and $\{y_n\}$ be two sequences of $X$. If either $\{x_n\}$ or $\{y_n\}$
 is bounded and $\phi(x_n,y_n)\to0$ as $n\to \infty$, then $\|x_n-y_n\|\to0$ as $n\to\infty$.\label{bd}
\end{lemma}

\begin{lemma} \cite{b1}\label{lem1} \rm
Let $C$ be a nonempty closed and convex subset of a smooth, strictly convex and reflexive Banach space $E$. Then, the following are equivalent.\\
$(i)$ $C$ is a sunny generalized nonexpansive retract of $E$,\\
$(ii)$ $C$ is a generalized nonexpansive retract of $E$,\\
$(iii)$ $JC$ is closed and convex.
\end{lemma}

\begin{lemma} \cite{b1}\label{lem2} \rm
Let $C$ be a nonempty closed and convex subset of a smooth and  strictly convex Banach space $E$ such that there exists a sunny generalized 
nonexpansive retraction $R$ from $E$ onto $C$. Then, the following hold.\\
$(i)$ $z=Rx$ iff $\langle x-z, Jy-Jz\rangle \leq 0$ for all $y\in C$,\\
$(ii)$ $\phi (x,Rx) + \phi (Rx, z) \leq \phi (x,z)$.  
\end{lemma}

\begin{lemma} \cite{itaka1}\label{ww1} \rm
 Let $C$ be a nonempty closed sunny generalized nonexpansive retract of a smooth and strictly convex Banach space $E$. Then the sunny generalized nonexpansive
 retraction from $E$ to $C$ is uniquely determined.
\end{lemma}

\begin{lemma} \cite{Blum}\label{ww3} \rm
 Let $C$ be a nonempty closed subset of a smooth, strictly convex and reflexive Banach space $E$ such that $JC$ is closed and convex, let $f$ be a bifunction
 from $JC\times JC$ to $\mathbb{R}$ satisfying $(A1)-(A4)$. For $r>0$ and let $x\in E$. Then there exists $z\in C$ such that
 $f(Jz,Jy)+\frac{1}{r}\langle z-x, Jy-Jz\rangle\geq0,~~\forall~~y\in C.$
\end{lemma}

\begin{lemma} \cite{ttaka}\label{ww4} \rm
 Let $C$ be a nonempty closed subset of a smooth, strictly convex and reflexive Banach space $E$ such that $JC$ is closed and convex, let $f$ be a bifunction
 from $JC\times JC$ to $\mathbb{R}$ satisfying $(A1)-(A4)$. For $r>0$ and let $x\in E$, define a mapping $T_r(x):E\rightarrow C$ as follows:
 $$T_r(x)=\{z\in C:f(Jz,Jy)+\frac{1}{r}\langle y - z, Jz-Jx\rangle\geq0,~~\forall~~y\in C\}.$$ Then the following hold:
 \begin{itemize}
 \item[(i)]  $T_r$ is single valued;
 \item[(ii)]  for all $x,y\in E$, $\langle T_rx-T_ry, JT_rx-JT_ry\rangle\leq\langle x-y, JT_rx-JT_ry\rangle$;
 \item[(iii)] $F(T_r)=EP(f)$;
 \item[(iv)] $\phi (p,T_r(x)) + \phi (T_r(x), x) \leq \phi (p,x)$ for all $p\in F(T_r)$.
 \item[(v)]  $JEP(f)$ is closed and convex. 
\end{itemize}
\end{lemma}

\begin{lemma} \cite{uoo}\label{ww5} \rm
 Let $C$ be a nonempty closed subset of a smooth, strictly convex and reflexive Banach space $E$. Let $A:C\rightarrow E^*$ be a continuous monotone
mapping. For $r>0$ and let $x\in E$, define a mapping $F_r(x):E\rightarrow C$ as follows:
 $$F_r(x)=\{z\in C:\langle y-z, Az\rangle+\frac{1}{r}\langle y-z, Jz-Jx\rangle\geq0,~~\forall~~y\in C\}.$$ Then the following hold:
 \begin{itemize}
 \item[(i)]  $F_r$ is single valued;
 \item[(ii)]  for all $x,y\in E$, $\langle F_rx-T_ry, JF_rx-JF_ry\rangle\leq\langle x-y, JF_rx-JF_ry\rangle$;
 \item[(iii)] $F(F_r)=VI(C,A)$;
 \item[(iv)] $\phi (p,F_r(x)) + \phi (F_r(x), x) \leq \phi (p,x)$ for all $p\in F(F_r)$.
 \item[(v)]  $JVI(C,A)$ is closed and convex. 
\end{itemize}
\end{lemma}

\begin{lemma} \cite{uoo}\label{lemma1} \rm
Let $E$ be a uniformly convex and uniformly smooth  real Banach space with dual space $E^*$ and let $C$ be a closed subset of $E$ such that $JC$ is closed and convex. Let $T$ be a
generalized $J_{*}-$nonexpansive map from $C$ to $E^*$ such that $F_{J}(T) \neq \emptyset$, then $F_{J}(T)$ and $JF_{J}(T)$ are closed. Additionally, if $JF_{J}(T)$ is convex, then $F_{J}(T)$ is a sunny generalized nonexpansive retract of $E$.
\end{lemma}

\section{Main Results}
\noindent Let $E$ be a uniformly smooth and uniformly convex real Banach space  with dual space $E^*$ and let $C$ be a nonempty closed and convex subset of $E$ 
such that $JC$ is closed and convex. Let $f_{l}, l=1, 2, 3, ..., L$ be a family of bifunctions from $JC\times JC$ to $\mathbb{R}$ satisfying $(A1)-(A4)$, $T_{n}:C\rightarrow E^*, n=1, 2, 3, ...$ be an infinite family of generalized $J_{*}-$nonexpansive maps, and $A_{k}:C\rightarrow E^*, k=1, 2, 3, ..., N$ be a finite family of continuous monotone mappings. Let the sequence $\{x_{n}\}$ be generated by the following iteration process:

\begin{equation}\label{alg3.1}
\begin{cases}  & x_{1} = x\in C; C_{1}=C, \cr
& z_{n} := \{z\in C:f_n(Jz,Jy)+\frac{1}{r_n}\langle y-z, Jz-Jx_n\rangle\geq0,~~\forall~~y\in C\},\cr
& u_{n} := \{z\in C: \langle y - z, A_n z \rangle + \frac{1}{r_n}\langle y-z, Jz-Jx_n\rangle\geq 0,~~\forall~~y\in C\},\cr
                     & y_n =J^{-1}(\alpha_{1}Jx_{n} + \alpha_{2} Jz_n + \alpha_3  T_{n}u_{n}),\cr
                     
                    & C_{n+1} =\{z\in C_{n} : \phi(z, y_{n}) \leq \phi(z, x_{n})\},\cr
                     & x_{n+1} = R_{C_{n+1}}x,
\end{cases} 
\end{equation}

for all $n\in \mathbb{N},$ with $ \alpha_1, \alpha_2, \alpha_3\in (0,1)$ satisfying $\alpha_1 + \alpha_2 + \alpha_3 = 1$,  $\{r_n\}\subset [a,\infty)$
 for some $a>0$, $A_n = A_{n(mod~  N)}$ and $f_n(\cdot,\cdot) = f_{n(mod ~L)}(\cdot,\cdot)$.
\noindent 
\begin{lemma}
The sequence $\{x_{n}\}$ generated by (\ref{alg3.1}) is well defined.
\end{lemma}

\begin{proof}
Observe that $JC_1$ is closed and convex. Moreover, it is easy to see that $\phi(z,y_n)\leq \phi(z,x_n)$ is equivalent to 
$$0\leq ||x_n||^2 - ||y_n||^2-2\langle z, Jx_n - Jy_n\rangle,$$ which is affine in $z$. Hence, by induction $JC_n$ is closed and convex for each $n\geq1$. Therefore, from 
Lemma \ref{lem1}, we have that $C_n$ is a sunny generalized retract of $E$ for each $n\geq1$. This shows that $\{x_{n}\}$ is well defined.
\end{proof}

\begin{theorem}\label{main}
Let $E$ be a uniformly smooth and uniformly convex real Banach space  with dual space $E^*$ and let $C$ be a nonempty closed and convex subset of $E$ 
such that $JC$ is closed and convex. Let $f_{l}, l=1, 2, 3, ..., L$ be a family of bifunctions from $JC\times JC$ to $\mathbb{R}$ satisfying $(A1)-(A4)$, 
$T_{n}:C\rightarrow E^*, n=1, 2, 3, ...$ be an infinite family of generalized $J_{*}-$nonexpansive maps, $A_{k}:C\rightarrow E^*, k=1, 2, 3, ..., N$ be a finite family of continuous monotone mappings and $\Gamma$ be a family of 
$J_{*}-$closed and generalized $J_{*}-$nonexpansive maps from $C$ to $E^*$ such that $\cap_{n=1}^{\infty}F_{J}(T_{n})=F_{J}(\Gamma) \neq \emptyset$
and $B := F_{J}(\Gamma)\cap\Big[\cap_{l=1}^{L} EP(f_{l})\Big]\cap \Big[\cap_{k=1}^{N}VI(C,A_{k})\Big] \neq \emptyset.$ Assume that $JF_{J}(\Gamma)$ is convex and 
$\{T_{n}\}$ satisfies the NST-condition with $\Gamma$. Then, $\{x_{n}\}$ generated by (\ref{alg3.1}) converges strongly 
to $R_B x$, where $R_B$ is the sunny generalized nonexpansive retraction of $E$ onto $B$.  
\end{theorem}

\begin{proof}
 The proof is given in $6$ steps.
 
\noindent {\bf Step 1}: We show that the expected limit $R_B x$ exists as a point in $C_n \; \text{for all} ~ n \ge 1$.

\noindent 
First, we show that $B \subset C_n \; \text{for all} \; n \ge 1$ and $B$ is a sunny generalized retract of $E$. \\
Since $C_1=C$, we have $B \subset C_1$.
Suppose  $B \subset C_{n}$ for some $n\in \mathbb{N}$. Let $u\in B$.  
We observe from algorithm (\ref{alg3.1}) that $ u_n=F_{r_n}x_n$ and $ z_n=T_{r_n}x_n$ for all $n\in \mathbb{N}$, using this and the fact that $\{T_{n}\}$
is an infinite family of generalized $J_{*}-$nonexpansive maps, the definition of $y_n$, Lemmas \ref{ww4}, \ref{ww5}, and  \ref{zhang}, we compute as follows:
\begin{eqnarray}\label{222}
 \phi(u, y_{n}) & = & \phi(u, J^{-1}(\alpha_{1}Jx_{n} + \alpha_2Jz_n + \alpha_3T_{n}u_{n} )\nonumber\\
                & \leq & \alpha_1\big[||u||^{2} -2\langle u, Jx_{n}\rangle + ||x_{n}||^{2}\big] +
                \alpha_2\big[||u||^{2} -2\langle u, Jz_{n}\rangle + ||z_{n}||^{2}\big] \nonumber\\
                &&+ \alpha_3\big[||u||^{2} -2\langle u, J(J_{*}\circ T_{n})u_{n}\rangle  + ||T_{n}u_{n}||^{2}\big]\nonumber\\
                 && - \alpha_1\alpha_3g(||Jx_{n}- J(J_{*}\circ T_{n})u_{n}||)\nonumber\\
                 & = & \alpha_1\phi(u, x_{n}) + \alpha_2\phi(u, z_{n}) + \alpha_3\phi(u, (J_{*}\circ T_{n})u_{n}) - 
                 \alpha_1\alpha_3g(||Jx_{n}- T_{n}u_{n}||)\nonumber\\
                 & \le & \alpha_1\phi(u, x_{n}) + \alpha_2\phi(u, z_{n}) + \alpha_3\phi(u, u_n) - 
                 \alpha_1\alpha_3g(||Jx_{n}- T_{n}u_{n}||) \\
                 & = & \alpha_1\phi(u, x_{n}) + \alpha_2\phi(u, T_{r_n}x_n) + \alpha_3\phi(u, u_n) - 
                 \alpha_1\alpha_3g(||Jx_{n}- T_{n}u_{n}||)\nonumber\\
                 & \le & \alpha_1\phi(u, x_{n}) + \alpha_2\phi(u, x_n) + \alpha_3\phi(u, u_n) - 
                 \alpha_1\alpha_3g(||Jx_{n}- T_{n}u_{n}||),\nonumber
\end{eqnarray}
which yields 
\begin{equation}\label{key}
\phi(u, y_{n}) \leq \phi(u, x_{n})  - \alpha_1\alpha_3g(||Jx_{n} - T_{n}u_{n}||).
\end{equation}
Hence, $\phi(u, y_{n}) \leq \phi(u, x_{n})$  and we have that $u\in C_{n+1}$, which implies that $B \subset{C_n} $ for all $n \ge 1$.
Moreover,
From Lemma \ref{ww4} and \ref{ww5} both  $JVI(C,A_k)$ and $JEP(f_l)$ are closed and convex for each $l$ and for each $k$. Also, using our assumption and lemma \ref{lemma1}, we have that  $J(F_{J}(\Gamma)$ is closed and convex.  
Since $E$ is uniformly convex, $J$ is one-to-one. Thus, we have that,

\overfullrule = 0mm\hbox to 10pt {
$J\Big(F_{J}(\Gamma)\cap\Big[\cap_{l=1}^{L} EP(f_{l})\Big]\cap \Big[\cap_{k=1}^{N}VI(C,A_{k})\Big]\Big)= JF_{J}(\Gamma)\cap J\Big[\cap_{l=1}^{L} EP(f_{l})\Big] \cap J\Big[\cap_{k=1}^{N}VI(C,A_{k})\Big]$} 

\noindent  
so $J(B)$ is closed and convex. Using Lemma \ref{lem1}, we obtain that $B$ is a sunny generalized retract of $E$. Therefore, from Lemma \ref{ww1} , we have that $R_Bx$ exists as a point in $C_n$ for all $n \ge 1$. This completes step 1.

\noindent{\bf Step 2:} We show that the sequence $\{x_{n}\}$ defined by \eqref{alg3.1} converges to some $x^* \in C.$

\noindent Using the fact that $x_n=R_{C_n}x$ and Lemma \ref{lem2}$(ii)$, we obtain
$$\phi(x,x_{n})=\phi(x,R_{C_{n}}x)\leq \phi(x,u) - \phi( R_{C_{n}}x,u)\leq \phi(x,u),$$
for all $u\in F_{J}(\Gamma)\cap EP(f_l) \cap VI(C,A_k)\subset C_{n}; (l = 1,2, \dots , L; \; k = 1,2, \dots , K).$
This implies that $\{\phi(x,x_{n})\}$ is bounded. Hence, from equation \eqref{fi}, $\{x_{n}\}$ is bounded. Also, since  
 $x_{n+1} = R_{C_{n+1}}x \in C_{n+1} \subset C_n$, and $x_n=R_{C_n}x \in C_n$, applying Lemma \ref{lem2}$(ii)$ gives

$$\phi(x,x_{n})\leq \phi(x,x_{n+1}) ~\forall~ n\in \mathbb{N}.$$ So, $\lim _{n\rightarrow \infty}\phi(x, x_{n})$ exists.
Again, using Lemma \ref{lem2}$(ii)$ and $x_{n}=R_{C_{n}}x$, we obtain that for all $m,n\in \mathbb{N}$ with $m>n$, 
\begin{eqnarray}
\phi(x_{n},x_{m})& =&\phi(R_{C_{n}}x,x_{m})\leq  \phi(x,x_{m}) -\phi(x,R_{C_{n}}x) \nonumber\\
&= & \phi(x,x_{m}) -\phi(x,x_{n}) \rightarrow 0~as~ n\rightarrow \infty.
\end{eqnarray} From Lemma \ref{man}, we conclude that $||x_{n}-x_{m}||\rightarrow 0, ~as~ m, ~ n\rightarrow \infty.$ Hence,
$\{x_{n}\}$ is a Cauchy sequence in $C$, and so, there exists $x^*\in C$ such that $x_{n}\rightarrow x^*$ completing  step 2.


\noindent{\bf Step 3}: We prove $x^*\in \cap_{k = 1}^{N} VI(C,A_k)$.\\
From the definitions of $C_{n+1}$ 
and $x_{n+1}$, we obtain that $\phi(x_{n+1}, y_{n})\leq \phi(x_{n+1},x_{n})\rightarrow 0$ as $n\rightarrow \infty.$ Hence, by Lemma \ref{man} , we have that
\begin{equation}\label{lim}
\underset{n \to \infty}{lim}||x_{n}-y_{n}||= 0.    
\end{equation}

 \noindent Since from step 2 $x_n\rightarrow x^*~as~n\rightarrow \infty$, equation (\ref{lim}) implies that $y_n\rightarrow x^*~as~n\rightarrow \infty$.   Using the fact that $u_n=F_{r_n}x_n$ for all $n\in \mathbb{N}$ and Lemma \ref{man}, we get for $u \in B,$
\begin{eqnarray}\label{new2}
 \phi(u_{n}, x_n)  &=& \phi(F_{r_n}x_n, x_n)\\
                  &\leq&\phi(u, x_n) - \phi(u, F_{r_n}x_n)\nonumber\\ 
                  &=&\phi(u, x_n) - \phi(u, u_n).\nonumber 
                   \end{eqnarray}
                   
\noindent From equations \eqref{222} and \eqref{key} we have
\begin{equation}\label{key2}
\phi(u, y_n) \le \alpha_1\phi(u, x_{n}) + \alpha_2\phi(u, x_n) + \alpha_3\phi(u, u_n)  \le \phi(u, x_n).    
\end{equation}

\noindent Since $x_n, y_n \to x^*$ as $n \to \infty$, equation \eqref{key2} implies that $\phi(u, u_n) \to \phi(u,x^*) $ as $n \to \infty$. Therefore, from \eqref{new2}, we have $\phi(u, x_n) - \phi(u, u_n) \to 0$ as $n \to \infty$ which implies that  $\lim_{n\rightarrow \infty}\phi(u_n,x_n)=0.$ Hence, from Lemma \ref{man},
we have
\begin{equation}\label{limun}
\lim_{n\rightarrow \infty}||u_n-x_n||=0.     
\end{equation}

\noindent Observe that since $J$ is uniformly continuous on bounded subsets of $E$, it follows from (\ref{limun}) that $||Ju_{n}-Jx_{n}||\rightarrow 0.$
  
\noindent Again, since $r_n\in [a,\infty),$ we have that 
\begin{eqnarray}\label{2213}
 \lim_{n\rightarrow \infty}\frac{||Ju_n-Jx_n||}{r_n}=0.
\end{eqnarray}
 
\noindent From $u_n=F_{r_n}x_n$, we have
\begin{eqnarray}\label{221}
 \langle y-u_n, A_nu_n\rangle+\frac{1}{r_n}\langle y -u_n, Ju_n-Jx_n\rangle\geq0,~~\forall~~y\in C.
\end{eqnarray}
Let $\{n_l\}_{l = 1}^{\infty} \subset{\mathbb{N}} $ be such that $A_{n_{l}} = A_1 \; \forall \; l \ge 1.$ Then, from (\ref{221}), we obtain 
\begin{eqnarray}\label{2211}
 \langle y-u_{n_l}, A_1u_{n_l}\rangle+\frac{1}{r_{n_l}}\langle y -u_{n_l}, Ju_{n_l}-Jx_{n_l}\rangle\geq0,~~\forall~~y\in C.
\end{eqnarray}
If we set $v_t=ty+(1-t)x^*$ for all $t\in(0,1]$ and $y\in C$, then we get that $v_t\in C.$ Hence, it follows from (\ref{2211}) that

\begin{eqnarray}
 \langle v_t-u_{n_l}, A_1u_{n_l}\rangle+\langle y -u_{n_l},\frac{Ju_{n_l}-Jx_{n_l}}{r_{n_l}} \rangle \geq0.
\end{eqnarray}

\noindent This implies that 
\begin{eqnarray*}
 \langle v_t-u_{n_l}, A_1v_t\rangle&\geq&\langle v_t-u_{n_l}, A_1v_t\rangle-\langle v_t-u_{n_l}, A_1u_{n_l}\rangle-\langle y -u_{n_l},\frac{Ju_{n_l}-Jx_{n_l}}{r_{n_l}} \rangle \\
 &=&\langle v_t-u_{n_l}, A_1v_t-A_1u_{n_l}\rangle-\langle y -u_{n_l},\frac{Ju_{n_l}-Jx_{n_l}}{r_{n_l}} \rangle.
\end{eqnarray*}
Since $A_1$ is monotone, $\langle v_t-u_{n_l}, A_1v_t-Au_{n_l}\rangle\geq0$. 
Thus, using (\ref{2213}), we have that
\begin{eqnarray*}
0\leq \lim_{l \rightarrow \infty} \langle v_t-u_{n_l}, A_1v_t\rangle = \langle v_t-x^*, A_1v_t\rangle,
\end{eqnarray*} 
therefore, 
\begin{eqnarray*}
\langle y-x^*, A_1v_t\rangle\geq0, ~~\forall~~y\in C.
\end{eqnarray*} 
Letting $t\rightarrow0$ and using continuity of $A_1$, we have that
\begin{eqnarray*}
\langle y-x^*, A_1x^*\rangle\geq0, ~~\forall~~y\in C.
\end{eqnarray*} 
This implies that $x^*\in VI(C,A_1)$. Similarly, if $\{n_i\}_{i = 1}^{\infty} \subset{\mathbb{N}} $ is such that $A_{n_{i}} = A_2 \; ~ \text{for all} ~ \; i \ge 1$, then we have again that $x^*\in VI(C,A_2)$. If we continue in similar manner, we obtain that $x^*\in \cap_{k = 1}^{N} VI(C,A_k).$

\noindent{\bf Step 4}: We prove that $x^*\in F_{J}(\Gamma)$.\\
First, we show that  $\lim _{n\rightarrow \infty}||Jx_{n}-Tu_{n}||=0 ~\forall~ T\in \Gamma$.\\ 
From inequality (\ref{key}) and the fact that $g$ is nonnegative, we obtain
$$0 \le \alpha_1\alpha_3g(||Jx_{n}-T_{n}u_{n}||) \leq \phi(u,x_{n})-\phi(u,y_{n})\leq 2||u||.||Jx_{n}-Jy_{n}|| + ||x_{n}-y_{n}||M,$$
 
\noindent for some $M > 0.$ 
\noindent Thus, using (\ref{lim}) and properties of $g$, we obtain that\\ $\lim_{n\rightarrow \infty}||Jx_{n}-T_{n}u_{n}||\ = 0$. Using the above and triangle inequality gives $\|Ju_n -T_nu_n|\ \to 0 ~~as~~ n \to \infty.$  Since $\{T_{n}\}_{n=1}^{\infty}$ satisfies the NST condition with $\Gamma$, we have that
\begin{equation}
 \lim _{n\rightarrow \infty}||Ju_{n}-Tu_{n}||=0 ~\forall~ T\in \Gamma. 
\end{equation}
  
\noindent Now, from equation (\ref{limun}), we have
$u_{n}\rightarrow x^*\in C$. Assume that  $(J_{*}\circ T)u_{n}\rightarrow y^*$. Since $T$ is $J_{*}-$closed, we have $y^*=(J_{*}\circ T)x^*$.
Furthermore, by the uniform continuity of $J$ on bounded subsets of $E$, we have:
$Ju_{n} \rightarrow Jx^*$ and $J(J_{*}\circ T)u_{n}\rightarrow Jy^*$ as $n\rightarrow \infty.$ Hence, we have
$$ \lim_{n\rightarrow \infty}||Ju_{n}-J(J_{*}\circ T)u_{n}||  =  \lim_{n\rightarrow \infty}||Ju_{n}-Tu_{n}||=0,  ~\forall~ T\in \Gamma, $$  which implies
$ ||Jx^{*}-Jy^{*}||= ||Jx^{*}-J(J_{*}\circ T)x^*||= ||Jx^*-Tx^*||=0 .$ So, $x^*\in F_{J}(\Gamma)$.


\noindent{\bf Step 5}: We prove that $x^*\in\cap_{l=1}^{L} EP(f_{l})$.\\
This follows by similar argument as in step 3 but for the sake of completeness we provide the details. Using the fact that $z_n=T_{r_n}x_n$ and Lemma \ref{ww4}, 
we obtain that for $u\in F_{J}(\Gamma)\cap EP(f_l) \cap VI(C,A_k) \; \text{for all} ~ i,k,$ 
\begin{eqnarray}\label{new3}
 \phi(z_n , x_n)  &=& \phi(T_{r_n}x_n, x_n)\\
                  &\leq&\phi(u, x_n) - \phi(u, T_{r_n}x_n)\nonumber\\
                  &=&\phi(u, x_n) -  \phi(u,z_n).  \nonumber
\end{eqnarray}
                   
\noindent From equations \eqref{222} and \eqref{key}, we have
\begin{equation}\label{key3}
\phi(u, y_n) \le \alpha_1\phi(u, x_{n}) + \alpha_2\phi(u, z_n) + \alpha_3\phi(u, x_n)  \le \phi(u, x_n).    
\end{equation}

\noindent Since $x_n,\; y_n,\; u_n \to x^*$ as $n \to \infty$, from equation \eqref{key3} we have $\phi(u, z_n) \to \phi(u,x^*) $ as $n \to \infty$. Therefore, from \eqref{new3}, we have $\phi(u, x_n) - \phi(u, u_n) \to 0$ as $n \to \infty$. Hence $\lim_{n\rightarrow \infty}\phi(z_n,x_n)=0.$ 
 From Lemma \ref{man},
we have
\begin{equation}\label{limzn}
 \lim_{n\rightarrow \infty}||z_n- x_n||=0,   
\end{equation}
which implies that $z_n \to x^* ~~as~~ n \to \infty.$ Again, since $J$ is uniformly continuous on bounded subsets of $E$, (\ref{limzn}) implies $\|Jz_n - Jx_n \| \to 0$. Since $r_n\in [a,\infty),$ we have that 
\begin{eqnarray}\label{223}
 \lim_{n\rightarrow \infty}\frac{||Jz_n - Jx_n||}{r_n}=0.
\end{eqnarray}

\noindent Since $z_n=T_{r_n}x_n$, we have that 
$$\frac{1}{r_n}\langle y - z_n, Jz_n -Jx_n\rangle\geq - f_n(Jz_n,Jy),~~\forall~~y\in C.$$
Let $\{n_l\}_{l = 1}^{\infty} \subset{\mathbb{N}} $ be such that $f_{n_{l}} = f_1 \; \forall \; l \ge 1.$ Then, using (A2), we have 

\begin{eqnarray}\label{224}
 \langle y - z_n, \frac{Jz_n -Jx_n}{r_n}\rangle \geq-f_1(Jz_n,Jy)\geq f_1(Jy,Jz_n),~~\forall~~y\in C.
\end{eqnarray}
Since $f_1(x,\cdot)$ is convex and lower-semicontinuous and $z_n\rightarrow x^*$, it follows from equation (\ref{223}) and inequality (\ref{224}) that
$$f_1(Jy,Jx^*)\leq0,~~\forall~~y\in C.$$
For $t\in(0,1]$ and $y\in C$, let $y^*_t=tJy+(1-t)Jx^*$. Since $JC$ is convex, we have that $y^*_t\in JC$ and hence $f_1(y^*_t,Jx^*)\leq0$. 
Applying (A1) gives,
$$0=f_1(y^*_t,y^*_t)\leq tf_1(y^*_t,Jy)+(1-t)f_1(y^*_t,Jx^*)\leq tf_1(y^*_t,Jy),~~\forall~~y\in C.$$ This implies that
$$f_1(y^*_t,Jy)\geq 0,~~\forall~~y\in C.$$
Letting  $t\downarrow0$ and using (A3), we get 
$$f_1(Jx^*,Jy)\geq 0,~~\forall~~y\in C.$$
Therefore, we have that $Jx^*\in JEP(f_1).$ This implies that $x^*\in EP(f_1).$  Applying similar argument, we can show that $x^*\in EP(f_l)$ for $l = 2,3, \dots , L.$ Hence, $x^*\in\cap_{l=1}^{L} EP(f_{l}).$

\noindent{\bf Step 6}: Finally, we show that $x^* = R_{B}x.$\\
From Lemma \ref{lem2}$(ii)$, we obtain that
\begin{equation}\label{eq4.3}
\phi(x,R_Bx) \leq \phi(x,x^*) - \phi(R_Bx, x^*) \leq  \phi(x,x^*).
\end{equation}

\noindent Again, using Lemma \ref{lem2}$(ii)$, definition of $x_{n+1}$, and $ x^*\in B \subset C_{n},$ we compute as follows:
\begin{eqnarray}
 \phi(x,x_{n+1}) &\leq &\phi(x,x_{n+1}) + \phi(x_{n+1},R_Bx)\nonumber\\
                  & = & \phi( x, R_{C_{n+1}}x) +  \phi(R_{C_{n+1}}x, R_Bx) \leq \phi(x, R_Bx).\nonumber
 \end{eqnarray}
Since $x_{n}\rightarrow x^*$, taking limits on both sides of the last inequality, we obtain
\begin{equation}\label{eq4.4}
\phi(x,x^*) \leq \phi(x, R_Bx).
\end{equation}
Using inequalities (\ref{eq4.3}) and (\ref{eq4.4}), we obtain that $\phi(x,x^*) = \phi(x, R_Bx)$. By the 
uniqueness of $R_B$(Lemma \ref{ww1}), we obtain that $x^*=R_Bx$. This completes proof of the theorem. 
\end{proof}
\section{Applications}
\begin{corollary}\label{corollary1}
Let $E$ be a uniformly smooth and uniformly convex real Banach space  with dual space $E^*$ and let $C$ be a nonempty closed and convex subset of $E$ 
such that $JC$ is closed and convex. Let $f$ be a bifunction from $JC\times JC$ to $\mathbb{R}$ satisfying $(A1)-(A4)$, $A:C\rightarrow E^*,$ be a continuous monotone mapping, $T:C\rightarrow E^*,$ be a generalized $J_{*}-$nonexpansive and $J_{*}-$closed map such that $B := F_{J}(T)\cap EP(f) \cap VI(C,A) \neq \emptyset.$ Assume that $JF_{J}(T)$ is convex. Then, $\{x_{n}\}$ generated by (\ref{alg3.1}) converges strongly 
to $R_B x$, where $R_B$ is the sunny generalized nonexpansive retraction of $E$ onto $B$.
\end{corollary}

\begin{proof}
Set $T_n := T$ for all $n\in \mathbb{N}$, $A := A_i$ for any $i = 1, 2, \cdots , N$, and $f := f_l$ for any $l = 1, 2, \cdots , L$. Then, from remark \ref{rmk1}, $\{T_{n}\}$ satisfies the NST-condition with $\{T\}$. The conclusion follows from Theorem \ref{main}.
\end{proof}

\begin{corollary}\label{corollary2}
Let $E$ be a uniformly smooth and uniformly convex real Banach space  with dual space $E^*$ and let $C$ be a nonempty closed and convex subset of $E$ 
such that $JC$ is closed and convex. Let $f_{l}, l=1, 2, 3, ..., L$ be a family of bifunctions from $JC\times JC$ to $\mathbb{R}$ satisfying $(A1)-(A4)$, 
$T_{n}:C\rightarrow E^*, n=1, 2, 3, ...$ be an infinite family of generalized $J_{*}-$nonexpansive maps and $\Gamma$ be a family of 
$J_{*}-$closed and generalized $J_{*}-$nonexpansive maps from $C$ to $E^*$ such that $\cap_{n=1}^{\infty}F_{J}(T_{n})=F_{J}(\Gamma) \neq \emptyset$
and $B := F_{J}(\Gamma)\cap\Big[\cap_{l=1}^{L} EP(f_{l})\Big] \neq \emptyset.$ Assume that $JF_{J}(\Gamma)$ is convex and 
$\{T_{n}\}$ satisfies the NST-condition with $\Gamma$. Then, $\{x_{n}\}$ generated by (\ref{alg3.1}) converges strongly 
to $R_B x$, where $R_B$ is the sunny generalized nonexpansive retraction of $E$ onto $B$. 
\end{corollary}

\begin{proof}
Setting $A_k=0$ for any $k=1, 2, 3, ..., N$, then result follows from Theorem \ref{main}.
\end{proof}

\begin{remark}
\noindent We note here that the theorem and corollaries presented above are applicable in classical Banach spaces, such as $L_{p},~l_{p}, or~ W^{m}_{p}(\Omega), 1<p<\infty$, where $W^{m}_{p}(\Omega)$ denotes the usual Sobolev space.
\end{remark}

\begin{remark}\label{remduality}(\cite{yak}; p. 36)
The analytical representations of duality maps are known in a number of Banach spaces, for example, in the spaces $L_p,$ $l_p,$ and
$W^p_m(\Omega),$ $p \in (1,\infty)$, $p^{-1}+q^{-1}=1$.
\end{remark}

\begin{corollary}\label{corollary6}
Let $E=H$, a real Hilbert space and let $C$ be a nonempty closed and convex subset of $H$. Let $f_{l}, l=1, 2, 3, ..., L$ be a family of bifunctions from $C\times C$ to $\mathbb{R}$ satisfying $(A1)-(A4)$, 
$T_{n}:C\rightarrow H, n=1, 2, 3, ...$ be an infinite family of nonexpansive maps, $A_{k}:C\rightarrow H, k=1, 2, 3, ..., N$ be a finite family of continuous monotone mappings and $\Gamma$ be a family of 
closed and generalized nonexpansive maps from $C$ to $H$ such that $\cap_{n=1}^{\infty}F(T_{n})=F(\Gamma) \neq \emptyset$
and $B := F(\Gamma)\cap\Big[\cap_{l=1}^{L} EP(f_{l})\Big]\cap \Big[\cap_{k=1}^{N}VI(C,A_{k})\Big] \neq \emptyset.$ Assume that $\{T_{n}\}$ satisfies the NST-condition with $\Gamma$.
Let $\{x_{n}\}$ be generated by:  

\begin{equation}\label{algcoro5}
\begin{cases}  & x_{1} = x\in C; C_{1}=C, \cr
& z_{n} := \{z\in C:f_n(z,y)+\frac{1}{r_n}\langle y - z, z - x_n\rangle\geq0,~~\forall~~y\in C\},\cr
& u_{n} := \{z\in C: \langle y - z, A_n z \rangle + \frac{1}{r_n}\langle y - z, z - x_n\rangle\geq 0,~~\forall~~y\in C\},\cr
                     & y_n =\alpha_{1}Jx_{n} + \alpha_{2} z_n + \alpha_3  T_{n}u_{n},\cr
                     
                    & C_{n+1} =\{z\in C_{n} : ||z - y_{n}|| \leq ||z - x_{n}||\},\cr
                     & x_{n+1} = P_{C_{n+1}}x,
\end{cases} 
\end{equation}

for all $n\in \mathbb{N}, \; \alpha_1, \alpha_2, \alpha_3\in (0,1)$ such that $\alpha_1 + \alpha_2 + \alpha_3 = 1$,  $\{r_n\}\subset [a,\infty)$
 for some $a>0$, $A_n = A_{n(mod~  N)}$ and $f_n(\cdot,\cdot) = f_{n(mod ~L)}(\cdot,\cdot)$. Then, $\{x_{n}\}$ converges strongly 
to $P_B x$, where $P_B$ is the metric projection of $H$ onto $B$. 
\end{corollary}

\begin{proof}
In a Hilbert space, $J$ is the identity operator and $\phi(x,y)=||x-y||^{2} ~ \text{for all} ~ x,y\in H$. The result follows from Theorem \ref{main}.
\end{proof}

\begin{example}
Let $E = l_p$, $1< p <\infty$, $\frac{1}{p} + \frac{1}{q} = 1$, and $C = \overline{B_{l_p}}(0,1)$ = $\{x \in l_p : ||x||_{l_p}\leq 1\}$. Then $JC = \overline{B_{l_q}}(0,1)$. Let $f : JC\times JC \longrightarrow \mathbb{R}$ defined by $f(x^*, y^*) = \langle J^{-1}x^*, x^* - y^*\rangle$ $\forall$ $x^* \in JC$, $A : C \longrightarrow l_q$ defined by $Ax = J(x_1, x_2, x_3, \cdots)$ $\forall$ $x = (x_1, x_2, x_3, \cdots) \in C$, $T : C \longrightarrow l_q$ defined by $Tx = J(0, x_1, x_2, x_3, \cdots)$ $\forall$ $x = (x_1, x_2, x_3, \cdots) \in C$, and $T_n : C \longrightarrow l_q$ defined by $T_{n}x = \alpha_nJx + (1 - \alpha_n)Tx, ~\forall n \geq 1,~\forall~ x\in C, \alpha_n \in (0,1) \text{ such that } 1 - \alpha_n \ge \frac{1}{2}$. Then $C$, $JC$, $f$, $A$, $T$, and $T_n$ satisfy the conditions of Corollary 4.1. Moreover, $0 \in F_{J}(\Gamma) \cap EP(f) \cap  VI(C,A)$.
\end{example}

\section{Conclusion}
\noindent Our theorem and its applications complement, generalize, and extend results of Uba et al. \cite{uoo}, Zegeye and Shahzad \cite{zesh}, Kumam \cite{kuman}, Qin and Su \cite{Qsu}, and Nakajo and Takahashi \cite{nakajok}. Theorem \ref{main} is a complementary analogue and extension of Theorem 3.2 of \cite{zesh} in the following sense: while Theorem 3.2 of \cite{zesh} is proved for a finite family of {\it self-maps} in uniformly smooth and strictly convex real Banach space which has the Kadec–Klee property, Theorem \ref{main} is proved for countable family of {\it non-self maps} in uniformly smooth and uniformly convex real Banach space; in Hilbert spaces, Corollary \ref{corollary6} is an extension of Corollary 3.5 of \cite{zesh} from {\it finite family of nonexpansive self-maps} to {\it countable family of nonexpansive non-self maps}. Additionally, Theorem \ref{main} extends and generalizes Theorem 3.7 of \cite{uoo} in the following sense: while Theorem 3.7 of \cite{uoo} studied equilibrium problem and countable family of {\it generalized $J_*-$nonexpansive non-self maps}, Theorem \ref{main} studied finite family of equilibrium and variational inequality problems and countable family of {\it generalizes $J_*-$nonexpansive non-self maps}; corollary \ref{corollary2} generalized Theorem 3.7 of \cite{uoo} to a finite family of equilibrium problems and countable family of {\it generalized $J_*-$nonexpansive non-self maps}. Furthermore, Corollary \ref{corollary1} extends Theorem 3.1 of \cite{kuman} from Hilbert spaces to a more general uniformly smooth and uniformly
convex Banach spaces and to a more general class of continuous monotone mappings.
Finally, Corollary \ref{corollary1} improves and extends the results in \cite{Qsu, nakajok} from {\it a nonexpansive self-map} to {\it a generalized $J_*-$nonexpansive non-self map}.

\end{document}